\documentclass[12pt]{amsart}

\usepackage{amssymb}

\usepackage{enumitem}

\usepackage{graphicx}

\makeatletter
\@namedef{subjclassname@2020}{%
  \textup{2020} Mathematics Subject Classification}
\makeatother

\usepackage[T1]{fontenc}
\usepackage{graphicx}
\usepackage{amssymb}
\usepackage{amsmath}
\usepackage{amsthm}
\usepackage[all]{xy}
\usepackage{ragged2e}
\usepackage{makecell}
\usepackage{enumitem}
\usepackage{color}

\newtheorem{thm}{Theorem}[section]
\newtheorem{cor}[thm]{Corollary}
\newtheorem{lem}[thm]{Lemma}

\newtheorem{exm}[thm]{Example}
\newtheorem{prop}[thm]{Proposition}
\theoremstyle{definition}
\newtheorem{defn}[thm]{Definition}
\theoremstyle{remark}
\newtheorem{rem}[thm]{\bf Remark}

\newcommand{\mo}{\mbox{-{\rm mod}}}
\newcommand{\proj}{\mbox{-{\rm proj}}}
\newcommand{\db}{\mathbf{D}^{\mathrm{b}}}
\newcommand{\kb}{\mathbf{K}^{\mathrm{b}}}
\newcommand{\cm}{\mathbf{C}}
\newcommand{\C}{\mathbf{C}}

\numberwithin{equation}{section}

\begin{document}


\baselineskip=17pt


\title[Derived Representation Type and Field Extensions]{Derived Representation Type and Field Extensions}

\author[J. Li]{Jie Li}
\address{Department of Mathematics, 
	University of Science and Technology of China, Hefei, Anhui, PR China}
\email{lijie0@mail.ustc.edu.cn}

\author[C. Zhang]{Chao Zhang}
\address{Department of Mathematics,
School of Mathematics and Statistics,
Guizhou University,
550025, Guiyang,
PR China}
\email{zhangc@amss.ac.cn}

\date{}

\begin{abstract}
Let $A$ be a finite-dimensional algebra over a field $k$. We define $A$ to be $\C$-dichotomic if it has the dichotomy property of the representation type on the category of certain bounded complexes of projective $A$-modules. If $k$ admits a finite separable field extension $K/k$ such that $K$ is algebraically closed, the real number field for example, we prove that $A$ is $\C$-dichotomic. As a consequence, the second derived Brauer-Thrall type theorem holds for $A$, i.e., $A$ is either derived-discrete or strongly derived-unbounded.
\end{abstract}



\maketitle

\section{Introduction}
The representation type is an important topic in representation theory of algebras, which studies the classification and distribution of the indecomposable modules. The most stimulating problems, the classical Brauer-Thrall type conjectures, were formulated for finite-dimensional $k$-algebras; see \cite{Br41, Th47, Jans57}. The first Brauer-Thrall type conjecture says that, an algebra either is of finite representation type or admits modules of arbitrary large dimensions. The second one states that any algebra is either of finite representation type or of strongly unbounded representation type. Two conjectures were proved for algebras over infinite perfect fields, see \cite{ASS06, Aus74, NR75, Ro68, Ro78, Rin80} and references therein. 

The Brauer-Thrall type theorems also relate with the celebrated tame-wild dichotomy theorem first proved by Drozd for finite-dimensional modules over finite-dimensional $K$-algebras over an algebraically closed field $K$; see \cite{Dro86}. The tame-wild dichotomy theorem is generalized for Cohen-Macauley modules in \cite{DG92} and for a class of bimodule matrix problems (over not only algebraically closed fields) in \cite{Sim97,Sim05}.

For algebras over algebraically closed fields, the derived representation type was pioneered by Vossieck \cite{Vo01}. He defined the derived-discrete algebras and classified them into two types: the algebras derived equivalent to hereditary algebras of finite type, and gentle one-cycle algebras not satisfying the {\it clock condition}. To study the Brauer-Thrall type theorems for derived categories, in \cite{ZH16}, some numerical invariants were introduced and strongly derived-unbounded algebras were defined naturally by Zhang and Han \cite{ZH16}. Then they prove in \cite{ZH16} the two derived Brauer-Thrall type theorems for algebras over algebraically closed fields. Moreover, the first one is proved by Zhang in \cite{Zh16b} for arbitrary Artin algebras. We mention that the tame-wild dichotomy for bounded derived categories of finite-dimensional algebras is established in \cite{BD03}.

The second derived Brauer-Thrall type theorem says that any finite-dimensional algebra $A$ has dichotomy property on the level of derived category, i.e., $A$ is either derived-discrete or strongly derived-unbounded. Strongly derived-unbounded algebras are studied more extensively, and dichotomy properties of representation types on levels of complex category and homotopy category for finite-dimensional algebras are obtained \cite{Zh16a}. Note that the dichotomy properties on three levels, including the second derived Brauer-Thrall type theorem, rely heavily on the classification of derived-discrete algebras, where the base field is required to be algebraically closed.

It is natural to ask if the second derived Brauer-Thrall type theorem and other dichotomy properties still hold for algebras over arbitrary infinite fields. Since the classification of derived-discrete algebras over arbitrary fields is unknown, our method is to consider if the properties are compatible under field extensions. Many properties about module categories, such as representation type \cite{JL82} and Auslander-Reiten theory \cite{Ka00}, are compatible under ground field extensions. In \cite{Li19}, derived-discreteness is proved to be compatible under finite separable field extensions.

In this paper, we revisit the notions about the representation type on complex categories, homotopy categories and derived categories for finite-dimensional algebras over arbitrary fields. Let $A$ be a finite-dimensional $k$-algebra. Denote by $\cm_m(A\proj)$ the category consisting of homotopy minimal complexes of projective $A$-modules which are concentrated between degree $0$ and $m$. An algebra $A$ is called $\C$-dichotomic if either $\cm_m(A\proj)$ is of finite representation type for each $m\geq 1$ or $\cm_{m'}(A\proj)$ is of strongly unbounded type for some $m'\geq 1$.

We then prove that $\C$-dichotomic algebras are preserved under finite separable field extensions. Making use of the dichotomy theorem (see \cite[Corollary 2.9]{Zh16a}) for algebras over algebraically closed fields, we obtain the following main theorem.

\medskip
\noindent {\bf Main Theorem.} \textit{Let $A$ be a finite-dimensional $k$-algebra and $K/k$ be a finite separable field extension such that $K$ is algebraically closed. Then $A$ is $\C$-dichotomic.}

\medskip
Since the $\C$-dichotomy implies the dichotomy properties of representation type on the levels of homotopy category and derived category (see Proposition \ref{dich-A-kb}),  we have

\medskip

\noindent {\bf The second derived Brauer-Thrall type theorem.} \textit{If $A$ is a finite-dimensional $k$-algebra and $K/k$ is a finite separable extension such that $K$ is algebraically closed, then $A$ is either derived-discrete or strongly derived-unbounded.}

\medskip

As an example, if $k$ is the real number field $\mathbb{R}$, then the main theorem holds. In particular, the second Brauer-Thrall type theorem is true for any finite-dimensional $\mathbb{R}$-algebra.

\section{The derived representation type of algebras}

In this section, we recall the definitions related to derived representation
types, and then introduce C-dichotomic algebras.

\subsection{Definitions Related to Derived Representation Type}

Let $k$ be an infinite field and $A$ a finite-dimensional algebra over $k$. Denote by $A\mo$ the category of all
finite-dimensional left $A$-modules and $A\proj$ its full
subcategory consisting of all finitely generated projective left
$A$-modules. Denote by $\cm^b(A\mo)$ the category of all bounded complexes of $A\mo$, and by $\cm^\mathrm{b}(A\proj)$ its full subcategory consisting of all bounded
complexes of $A\proj$. Denote by $\kb(A\proj)$ the
homotopy category of $\cm^\mathrm{b}(A\proj)$, and by $\db(A\mo)$ the bounded derived category of $A\mo$.

Recall that a complex $X=(X^i, d^i) \in \cm^\mathrm{b}(A\proj)$ is said to
be {\it homotopy minimal} if $\mathrm{Im} d^i \subseteq \mathrm{rad} X^{i+1}$ for all $i \in \mathbb{Z}$. For any integer $m\geq 0$, denote by
$\cm_m(A\proj)$ the full subcategory of $\cm^\mathrm{b}(A\proj)$ consisting of all homotopy minimal complexes $X=(X^i, d^i)$ such that $X^i=0$ for any
$i\notin\{0, 1, \cdots, m\}$.

We start this section by recalling from \cite{ZH16} definitions related to a finite derived representation type.

\begin{defn}
	Let $A$ be a $k$-algebra.\\
	(1) The category $\cm_m(A\proj)$ is defined to be {\it of finite representation type} if up to isomorphisms, there are only finitely many indecomposable objects in $\cm_m(A\proj)$.\\
	(2) The category $\kb(A\proj)$ is defined to be {\it discrete} if for each cohomology dimension vector, there admit only finitely many objects in $\kb(A\proj)$ up to isomorphisms. \\
	(3) The algebra $A$ is defined to be {\it derived-discrete} if for each cohomology dimension vector, there admit only finitely many objects in $\db(A\mo)$ up to isomorphisms.
\end{defn}

The next lemma shows the connections of the above definitions.

\begin{lem} \label{lem-discrete}
	Let $A$ be a $k$-algebra. Consider the following statements:\\
    {\rm (1)} for each $m>0$, $\cm_m(A\proj)$ is of finite representation type,\\
    {\rm (2)} the category $\kb(A\proj)$ is discrete,\\
    {\rm (3)} $A$ is derived-discrete. \\
    Then we have {\rm (1)} $\Longrightarrow$ {\rm (2)} $\Longleftrightarrow$ {\rm (3)}.
\end{lem}
\begin{proof}
	For the first implication. By \cite[Lemma 2.5]{Li19}, we only need to prove that, for each dimension vector $(n_i)_{i\in\mathbb{Z}}\in\mathbb{N}^{(\mathbb{Z})}$, the set  $$\{X=(X^i,d^i)\in\kb(A\proj)\;|\;\dim_kX_i=n_i,\forall i\in\mathbb{Z}\}$$ has finitely many isomorphism classes. After some shifts, there is an integer $m>0$ such that $n_i=0$ for $i<0$ and $i>m$. By assumption, $\cm_m(A\proj)$ has finitely many isomorphism classes. Each object $X$ in the above set is homotopy equivalent to a homotopy minimal complex, denoted by $\bar{X}$, in $\cm_m(A\proj)$. Since $\bar{X}$ is isomorphic to $\bar{Y}$ in $\cm_m(A\proj)$ implies that $X$ is isomorphic to $Y$ in $\kb(A\proj)$, the above set has finitely many isomorphism classes.
	
	The equivalence of the second statement and the third one follows by applying \cite[Lemma 2.5]{Li19}.
\end{proof}

Following \cite{ZH16}, given a complex $X$ in $\cm^b(A\mo)$, the {\it cohomological range} of $X$
is defined as $$\mathrm{hr}_k(X) := \mathrm{hl}_k(X) \cdot \mathrm{hw}(X),$$ where
$$\mathrm{hl}_k(X) := \max\{\dim_k H^i(X) \; | \; i \in \mathbb{Z}\}$$
and $$\mathrm{hw}(X) := \max\{j-i+1 \; | \; H^i(X) \neq 0 \neq H^j(X)\}.$$

Now we recall from \cite{ZH16} the definitions related to an infinite derived representation type. Note that the definition (1) has an equivalent form using the dimension of complex; see \cite[Lemma 1.6]{Zh16a}.

\begin{defn}
	Let $A$ be a $k$-algebra.\\
	(1) We say $\cm_m(A\proj)$ is {\it of strongly unbounded type} if there is an increasing sequence $(r_i)_{i\in\mathbb{N}}\in\mathbb{N}^ \mathbb{N}$ such that for each $r_i$, up to isomorphisms, there are infinitely many indecomposable objects in $\cm_m(A\proj)$ with cohomological range $r_i$.\\
	(2) We say $\kb(A\proj)$ is {\it of strongly unbounded type}  if there is an increasing sequence $(r_i)_{i\in\mathbb{N}}\in\mathbb{N}^ \mathbb{N}$ such that for each $r_i$, up to shifts and isomorphisms, there are infinitely many indecomposable objects in $\kb(A\proj)$ with cohomological range $r_i$.\\
	(3) We say $A$ is {\it strongly derived-unbounded} if there is an increasing sequence $(r_i)_{i\in\mathbb{N}}\in\mathbb{N}^ \mathbb{N}$ such that for each $r_i$, up to shifts and isomorphisms, there are infinitely many indecomposable objects in $\db(A\mo)$ with cohomological range $r_i$.
\end{defn}

\begin{exm}
{\rm (1)} The representation-infinite algebras over infinite perfect fields are strongly derived-unbounded, since the truth of classical Brauer-Thrall conjecture~II for module category \cite{NR75,Rin80} and the embedding from the module category to derived category.

{\rm (2)} Let $k$ be an algebraically closed field, and $A$ be a gentle algebra with one cycle with clock condition or more than one cycle, then $A$ is strongly unbounded. Here, the clock condition means that the number of clockwise relations on the cycle equals to that of counterclockwise ones. Indeed, $A$ has generalized bands in these cases, see \cite{Rin97} for example, and then one can construct an increasing sequence $(r_i)_{i\in\mathbb{N}}\in\mathbb{N}^ \mathbb{N}$ and infinite many non-isomorphic indecomposables in derived category for each $r_i$ by the combinatorial description in \cite{BM03}.

{\rm (3)} The algebras which are derived equivalent to strongly derived-unbounded algebras,  are also strongly derived-unbounded, since the derived equivalences can be realized as tensor functors by two-sided tilting complexes, under which cohomological ranges can be controlled, see \cite[Prop. 4]{ZH16}.
\end{exm}

The next lemma shows the connections of the above definitions. They were essentially proved in \cite{Bau07,Zh16a}, where $k$ was supposed to be algebraically closed. Here we include a proof for an arbitrary infinite field $k$.
The following notion is needed in the proof.

 Let $\mathbf{K}^{-,b}(A\proj)$ be the homotopy category consisting of bounded-above complexes with bounded cohomologies. There is a well-known triangle equivalence $$p\colon \db(A\mo)\longrightarrow\mathbf{K}^{-,b}(A\proj),$$ sending $X$ to its projective resolution $pX$; see \cite{Wei95}. We can further assume that $pX$ is homotopy minimal. For each $P$ in $\mathbf{K}^-(A{\mbox{-\rm proj}})$, let $P_{\geq t}\in\mathbf{K}^b(A{\mbox{-\rm proj}})$ be the brutal truncation of $P$ at degree $t$.

\begin{lem}\label{impc-rel}
	Let $A$ be a $k$-algebra. Consider the following statements:\\
   {\rm (1)} there is an $m\geq 1$ such that $\cm_m(A\proj)$ is of strongly unbounded type, \\
   {\rm (2)} the category $\kb(A\proj)$ is of strongly unbounded type,\\
   {\rm (3)} the algebra $A$ is strongly derived-unbounded.\\
   Then we have {\rm (1)} $\Longrightarrow$ {\rm (2)} $\Longleftrightarrow$ {\rm (3)}.
\end{lem}
\begin{proof}
	For the first implication. By (1), there is an increasing sequence $(r_i)_{i\in\mathbb{N}}\in\mathbb{N}^ \mathbb{N}$ such that for each $r_i$, up to isomorphisms, there are infinitely many indecomposable objects in $\cm_m(A\proj)$ with cohomological range $r_i$ for some $m\geq 1$. Given two complexes in $\cm_m(A\proj)$, the property of homotopy minimality implies that they are isomorphic in $\kb(A\proj)$ if and only if they are isomorphic in $\cm_m(A\proj)$. In addition, a complex in $\cm_m(A\proj)$ is indecomposable if and only if it is indecomposable as a complex in $\kb(A\proj)$. Hence for each $r_i$, there are infinitely many indecomposable objects in $\kb(A\proj)$ with cohomological range $r_i$.
	

	Now we prove the equivalence of the second statement and the third one. The ``$\Longrightarrow$'' part holds because $\kb(A\proj)$ is a full subcategory of $\db(A\mo)$ by embedding into $\mathbf{K}^{-,b}(A\proj)$.
	
	For the ``$\Longleftarrow$'' part, by assumption, there is an increasing sequence $(r_i)_{i\in\mathbb{N}}\in\mathbb{N}^ \mathbb{N}$ such that for each $r_i$, up to shifts and isomorphisms, there are infinitely many indecomposable objects in $\db(A\mo)$ with cohomological range $r_i$. Let $\mathcal{X}_i$ be the set of complexes in $\db(A\mo)$ with cohomological range $r_i$ whose nonzero cohomologies concentrate between degree $1$ and $r_i$. Then up to shifts and isomorphisms, $\mathcal{X}_i$ has infinitely many objects.
	
	For each $i>0$ and each $X$ in $\mathcal{X}_i$, denote by $(pX)$ its homotopy minimal projective resolution and $(pX)_{\geq 0}$ the brutal truncation at degree $0$. Then $(pX)_{\geq 0}$ is an indecomposable object in $\cm_{r_i}(A\proj)$ with $\mathrm{hr}((pX)_{\geq 0})\geq r_i$. By \cite[Lemma 2.4]{Li19}, the set $$\{\mathrm{hr}((pX)_{\geq 0})\;|\;X\in\mathcal{X}_i\}$$ also has an upper bound. By \cite[Lemma 2.3]{Li19}, there is a positive number $s_i$ between $r_i$ and the upper bound above such that the set $$\{(pX)_{\geq 0}\in\kb(A\proj)\;|\;X\in\mathcal{X}_i\}$$ has infinitely many objects up to isomorphisms.
	
	Since $(r_i)_{i\in\mathbb{N}}$ is an increasing sequence and $s_i\geq r_i$ for each $i$, inductively we can pick an increasing subsequence $(s'_i)_{i\in\mathbb{N}}\in\mathbb{N}^ \mathbb{N}$ of $(s_i)_{i\in\mathbb{N}}$. So $(s'_i)_{i\in\mathbb{N}}\in\mathbb{N}^ \mathbb{N}$ is an increasing sequence such that, up to shifts and isomorphisms, there are infinitely many indecomposable objects in $\kb(A\proj)$ with cohomological range $s'_i$. This completes the proof.
	
    We should notice that the proof here does not imply the strongly unboundedness of $\cm_{m}(A\proj)$ for some $m$, since different set $\mathcal{X}_i$ belong to different $\cm_{m}(A\proj)$ and we can not find a uniformed integer $m$, such that for each $t_i$, up to isomorphisms, there are infinitely many indecomposable objects in $\cm_m(A\proj)$ with cohomological range $t_i$ for some increasing sequence $(t_i)_{i\in\mathbb{N}}\in\mathbb{N}^ \mathbb{N}$.
\end{proof}

\subsection{$\C$-dichotomic Algebras}

In this subsection, we introduce the definition of $\C$-dichotomic algebras and its relation with the second derived Brauer-Thrall type theorem.

By Lemma \ref{lem-discrete} and Lemma \ref{impc-rel}, we have the following proposition.

\begin{prop}\label{equ KandD} For a finite-dimensional $k$-algebra $A$, the following statements are equivalent:\\
	{\rm (1)} the category $\kb(A\proj)$ is either discrete or of strongly unbounded type,\\
	{\rm (2)} $A$ is either derived-discrete or strongly derived-unbounded.
\end{prop}

As in \cite{ZH16}, if one of the above statements holds, we say that {\it the second derived Brauer-Thrall type theorem} holds for $A$. The following proposition due to \cite{ZH16} shows that, for any finite-dimensional algebra over an algebraically closed field, the second derived Brauer-Thrall type theorem holds.

\begin{prop}
	Let $A$ be a finite-dimensional algebra over an algebraically closed field. Then $A$ is either derived-discrete or strongly derived-unbounded.
\end{prop}

\begin{defn}
	A finite-dimensional $k$-algebra $A$ is defined to be {\it $\C$-dichotomic} if either $\cm_m(A\proj)$ is of finite representation type for each $m\geq 1$, or $\cm_{m'}(A\proj)$ is of strongly unbounded type for some $m'\geq 1$.
\end{defn}

\begin{rem}
	The category $\cm_m(A\proj)$ is of strongly unbounded type for some integer $m=M\geq 1$, which is equivalent to that $\cm_m(A\proj)$ is of strongly unbounded type for all $m\geq M$ since $\cm_m(A\proj)\subseteq\cm_{m+1}(A\proj)$. Thus a $k$-algebra $A$ is $\C$-dichotomic if either $\cm_m(A\proj)$ is of finite representation type for any $m\geq 1$, or $\cm_{m}(A\proj)$ is of strongly unbounded type for almost all positive integer $m$.
\end{rem}

\begin{prop}\label{dich-A-kb}
	Assume that $A$ is a $\C$-dichotomic finite-dimensional $k$-algebra.\\
	{\rm (1)} The category $\kb(A\proj)$ is either discrete or of strongly unbounded type.\\
	{\rm (2)} The second derived Brauer-Thrall type theorem holds for $A$.
\end{prop}
\begin{proof}
    (1). By Lemma \ref{lem-discrete}, if $\kb(A\proj)$ is not discrete, then it is not true that $\cm_m(A\proj)$ is of finite representation type for each $m\geq 1$. So $\cm_{m'}(A\proj)$ is of strongly unbounded type for some $m'\geq 1$ by assumption. By Lemma \ref{impc-rel}, $\kb(A\proj)$ is of strongly unbounded type. 
    
    (2) follows from (1) and Proposition \ref{equ KandD}
\end{proof}

Note that our definitions in this section also make sense for algebras over finite fields. In this case, we give an example which shows that the converse of the proposition may be not true.

\begin{exm}
	Let $k$ be a finite field. For each finite-dimensional $k$-algebra $A$, there are finitely many morphisms between projective $A$-modules. Then by \cite[Lemma~2.5]{Li19}{\rm (}whose proof doesn't depend on the cardinal of $k${\rm )}, $\kb(A\proj)$ is always discrete.
	
	If $A$ is a hereditary algebra over $k$ which is not of finite representation type, then $\cm_1(A\proj)$ is not of finite representation type. However, for any $m\geq 1$ and any $r_i>0$, there are finitely many objects $(X^j,d^j)$ in $\cm_m(A\proj)$ with $\sum_{j=0}^{m}X^j=r_i$. By \cite[Lemma 1.6]{Zh16a}{\rm (}whose proof doesn't depend on the cardinal of $k${\rm )}, $\cm_m(A\proj)$ is not of strongly unbounded type for any $m\geq 1$. Hence $A$ is not $\C$-dichotomic.
\end{exm}

The $\C$-dichotomy implies not only the second derived Brauer-Thrall type theorem, but also the equivalence of discretenesses and strongly unbounded properties on three levels, as in the following corollary.

\begin{cor}\label{equ}
	Assume that $A$ is a $\C$-dichotomic finite-dimensional $k$-algebra. \\
	{\rm (a)} The following three conditions are equivalent:
	\begin{enumerate}[itemindent=1em]
		\item[\rm (a1)] $A$ is derived-discrete,
		\item[{\rm (a2)}] the category $\kb(A\proj)$ is discrete,
		\item[{\rm (a3)}] the category $\cm_m(A\proj)$ is of finite representation type for each $m\geq 1$.
	\end{enumerate}
    {\rm (b)} The following three conditions are equivalent:
    \begin{enumerate}[itemindent=1em]
    	\item [{\rm (b1)}] $A$ is strongly derived-unbounded,
    	\item [{\rm (b2)}] the category $\kb(A\proj)$ is of strongly unbounded type,
    	\item [{\rm (b3)}] the category $\cm_m(A\proj)$ is of strongly unbounded type for some $m\geq 1$.
    \end{enumerate}
\end{cor}
\begin{proof} We only prove (a), and (b) can be proved in the similar way.
	By Lemma~\ref{impc-rel}, to prove that the equivalence of three statements, it suffices to prove that the discreteness of $\db(A\mo)$ implies the finiteness of $\cm_m(A\proj)$ for any $m>1$. If not, then $\cm_{M}(A\proj)$ is of strongly unbounded type for some $M>1$ since $A$ is a $\C$-dichotomic $k$-algebra. Therefore $A$ is strongly derived-unbounded by Lemma~\ref{impc-rel}. This is ridiculous since any algebra can not be derived-discrete and strongly derived-unbounded by definition.	
\end{proof}

\begin{rem}
	Assume that $A$ is a finite-dimensional $k$-algebra.\\
	{\rm (1)} If the field $k$ is algebraically closed, then $A$ is $\C$-dichotomic; see \cite[Corollary 2.9]{Zh16a}.\\
	{\rm (2)} We don't know whether or not $A$ is $\C$-dichotomic if the field $k$ is not algebraically closed.
\end{rem}

\medskip

\section{Base field extensions}

In this section, we mainly explore the $\C$-dichotomy of a $k$-algebra with $k$ admitting a finite separable extension $K/k$ such that $K$ is an algebraically closed field.

Let $K/k$ be a finite separable field extension. It is well known that the algebra extension $A\rightarrow A\otimes_kK$ induces an adjoint pair $(-\otimes_kK,F)$ between $A\mo$ and $A\otimes_kK\mo$, where $$F\colon A\otimes_kK\mo \longrightarrow A\mo$$ is the restriction functor. These two functors are both exact, mapping projective modules to projective modules and radicals to radicals. So they extend in a natural manner to adjoint pairs between $\cm_m(A\proj)$ and $\cm_m(A\otimes_kK\proj)$. We still denote them by $(-\otimes_kK, F)$ for convenience.

Since both $-\otimes_kK$ and $F$ are separable functors, each complex $X$ is a direct summand of $F(X\otimes_kK)$ in $\cm_m(A\proj)$ and each complex $Y$ is a direct summand of $F(Y)\otimes_kK$ in $\cm_m(A\otimes_kK\proj)$; see \cite{Li19}. So we have the following lemmas.

\begin{lem}\label{relative}
	{\rm (1)} For each indecomposable object $X$ in $\cm_m(A\proj)$, there is an indecomposable direct summand $Y$ of $X\otimes_kK$ in $\cm_m(A\otimes_kK\proj)$ such that $X$ is a direct summand of $F(Y)$. \\
	{\rm (2)} For each indecomposable object $Y$ in $\cm_m(A\otimes_kK\proj)$, there is an indecomposable direct summand $X$ of $F(Y)$ in $\cm_m(A\proj)$ such that $Y$ is a direct summand of $X\otimes_kK$.
\end{lem}

\begin{lem}\label{fin-lim}
	Let $K/k$ be a field extension of degree $l$, and $A$ be a $k$-algebra. \\
	{\rm (1)} For each indecomposable object $X$ in $\cm_m(A\proj)$, $X\otimes_kK$ has at most $l$ indecomposable direct summands up to isomorphisms.\\
	{\rm (2)}  For each indecomposable object $Y$ in $\cm_m(A\otimes_kK\proj)$, $F(Y)$ has at most $l$ indecomposable direct summands up to isomorphisms.
\end{lem}
\begin{proof}
	(1) Let $\{\alpha_1,\dots,\alpha_l\}$ be a $k$-basis of $K$. We have an isomorphism in $\cm_m(A\proj)$:
	\begin{align*}
	F(X\otimes_kK)&\simeq X^{\oplus l}\\
	x\otimes\lambda&\mapsto (\lambda_ix)_{i=1}^l,
	\end{align*}
	where $X^{\oplus l}$ is the direct sum of $l$ copies of $X$ and $\lambda_i$ are elements in $k$ such that $\lambda=\sum_{i=1}^{l}\lambda_i\alpha_i$. So our statement holds since $F$ is an additive functor.
	
	(2) For each indecomposable object $Y$ in $\cm_m(A\otimes_kK\proj)$, there is an indecomposable direct summand $X$ of $F(Y)$ in $\cm_m(A\proj)$ such that $Y$ is a direct summand of $X\otimes_kK$ (Lemma~\ref{relative}). Hence $F(Y)$ is a direct summand of $F(X\otimes_kK)$ in $\cm_m(A\proj)$. The isomorphism $F(X\otimes_kK)\simeq X^{\oplus l}$ then implies that $F(Y)$ has at most $l$ indecomposable direct summands.
\end{proof}

\begin{prop}\label{dich-ext}
	Let $A$ be a $k$-algebra and $K/k$ be a finite separable field extension. Then following statements hold.\\
	{\rm (1)} For each $m>0$, $\cm_m(A\proj)$ is of finite representation type if and only if so is $\cm_m(A\otimes_kK\proj)$.\\
	{\rm (2)} For each $m>0$, $\cm_m(A\proj)$ is of strongly unbounded type if and only if so is $\cm_m(A\otimes_kK\proj)$.

	As a consequence, $A$ is $\C$-dichotomic if and only if so is $A\otimes_kK$.
\end{prop}
\begin{proof}
	Let $l$ be the degree of the extension $K/k$.
	
    (1) The ``if'' part. For each $m\geq 1$, up to isomorphisms, let $Y_1,\dots,Y_n$ be all the indecomposable objects in $\cm_m(A\otimes_kK\proj)$. For each indecomposable $X$ in $\cm_m(A\proj)$, $X$ is a direct summand of $F(X\otimes_kK)$. By Lemma 3.1, there is an indecomposable object, say $Y_i$ for some $i\in\{1,\dots,n\}$, which is a direct summand of $X\otimes_kK$ such that $X$ is a direct summand of $F(Y_i)$. Since up to isomorphisms, there are only finitely many indecomposable direct summand of $\oplus_{i=1}^{n}F(Y_i)$, $\cm_m(A\proj)$ is of finite representation type.

    The ``only if'' part can be proved similarly.

    (2) For the ``if'' part, let $(r_i)_{i\in\mathbb{N}}\in\mathbb{N}^ \mathbb{N}$ be an increasing sequence such that for each $r_i$, the set $$\mathcal{X}_i=\{X\in\cm_m(A\proj)\;|\;X \mbox{ is indecomposable with } \mathrm{hr}_k(X)=r_i\}$$ has infinitely many objects up to isomorphisms (i.e. $\mathcal{X}_i$ has infinitely many isomorphism classes in $\cm_m(A\proj)$). Denote by $\mathcal{Y}_i$ all the indecomposable objects $Y$ in $\cm_m(A\otimes_kK\proj)$ such that $Y$ is a direct summand of $X\otimes_kK$ and $F(Y)$ contains $X$ as a direct summand for some $X$ in $\mathcal{X}_i$. Because $\mathcal{X}_i$ contains infinitely many objects up to isomorphisms, so does $\mathcal{Y}_i$.

    By the exactness of $-\otimes_kK$, we have $\mathrm{hr}_K(X\otimes_kK)=r_i$. By Lemma~\ref{fin-lim}~(1) $$\{\mathrm{hr}_K(Y)\;|\;Y\in \mathcal{Y}_i\}\subseteq[r_i/l,r_i].$$ So there is an integer $s_i$ between $r_i/l$ and $r_i$ such that, up to isomorphisms, there are infinitely many indecomposable objects in $\cm_m(A\otimes_kK\proj)$ with cohomological range $s_i$.

    Since $(r_i)_{i\in\mathbb{N}}$ is an increasing sequence, for each $i$ in $\mathbb{N}$ there is a larger $j$ in $\mathbb{N}$ such that $r_j>r_i/l$. Because $r_i\geq s_i\geq r_i/l$, we can pick inductively an increasing subsequence $(s'_i)_{i\in\mathbb{N}}\in\mathbb{N}^ \mathbb{N}$ of $(s_i)_{i\in\mathbb{N}}$ such that, up to isomorphisms, there are infinitely many indecomposable objects in $\cm_m(A\otimes_kK\proj)$ with cohomological range $s'_i$. Therefore $\cm_m(A\otimes_kK\proj)$ is of strongly unbounded type.

    The ``only if'' part. Let $(r_i)_{i\in\mathbb{N}}\in\mathbb{N}^ \mathbb{N}$ be an increasing sequence such that for each $r_i$, the set $$\mathcal{Y}_i=\{Y\in\cm_m(A\otimes_kK\proj)\;|\;Y \mbox{ is indecomposable with } \mathrm{hr}_K(Y)=r_i\}$$ has infinitely many objects up to isomorphisms. We denote by $\mathcal{X}_i$ all the indecomposable objects $X$ in $\cm_m(A\proj)$ such that $X$ is a direct summand of $F(Y)$ and $X\otimes_kK$ contains $Y$ as a direct summand for some $Y$ in $\mathcal{Y}_i$. Then $\mathcal{X}_i$ has infinitely many objects up to isomorphisms.

    Since $F$ is exact, $\mathrm{hr}_k(F(Y))=l\cdot r_i$. Here, we should notice that the cohomological range is defined by the dimension over $k$. By Lemma~\ref{fin-lim} ~(2), $$\{\mathrm{hr}_k(X)\;|\;X\in \mathcal{X}_i\}\subseteq[r_i, \;l\cdot r_i].$$ Then there is an integer $s_i$ between $r_i$ and $l\cdot r_i$ such that, up to isomorphisms, there are infinitely many indecomposable objects in $\cm_m(A\proj)$ with cohomological range $s_i$. 
    
    Since $(r_i)_{i\in\mathbb{N}}$ is an increasing sequence, for each $i$ in $\mathbb{N}$ there is a larger $j$ in $\mathbb{N}$ such that $r_j>l\cdot r_i$. So we can find inductively an increasing subsequence $(s'_i)_{i\in\mathbb{N}}\in\mathbb{N}^ \mathbb{N}$ of $(s_i)_{i\in\mathbb{N}}$ such that, up to isomorphisms, there are infinitely many indecomposable objects in $\cm_m(A\proj)$ with cohomological range $s'_i$. Therefore $\cm_m(A\proj)$ is of strongly unbounded type.
\end{proof}

\begin{rem}
The adjoint pairs $(-\otimes_kK,F)$ exist on the levels of homotopy category and derived category. By the argument used in the proof of Proposition~\ref{dich-ext}, we can prove that the second Brauer-Thrall type theorem holds for $A$ if and only if it holds for $A\otimes_kK$.
\end{rem}

If $K$ is algebraically closed, then by \cite[Corollary 2.9]{Zh16a} and the above proposition, we have

\begin{thm}\label{main}
	Let $A$ be a $k$-algebra and $K/k$ be a finite separable extension such that $K$ is algebraically closed. Then $A$ is $\C$-dichotomic.
\end{thm}

In view of Proposition \ref{dich-A-kb} and Corollary \ref{equ}, we conclude from Theorem \ref{main} the following useful corollary.
\begin{cor}
	Let $A$ be a $k$-algebra and $K/k$ be a finite separable extension such that $K$ is algebraically closed. Then the following statements hold.\\
	{\rm (1)} The second Brauer-Thrall type theorem holds for $A$.\\
	{\rm (2)} The category $\kb(A\proj)$ is either discrete or of strongly unbounded type.\\
	{\rm (3)} The algebra $A$ is derived-discrete if and only if $\kb(A\proj)$ is discrete if and only if $\cm_m(A\proj)$ is of finite representation type for each $m\geq 1$.\\
	{\rm (4)} The algebra $A$ is strongly derived-unbounded if and only if $\kb(A\proj)$ is of strongly unbounded type if and only if $\cm_m(A\proj)$ is of strongly unbounded type for some $m\geq 1$.
\end{cor}

\begin{cor}
	Let $A$ be a finite-dimensional algebra over the real number field. Then the second Brauer-Thrall type theorem holds for $A$.
\end{cor}

\begin{rem}
	{\rm (1)} In Theorem \ref{main}, the condition that $K/k$ is separable is necessary; see the example in \cite[Remark 3.4]{JL82}.\\
	{\rm (2)} It is still unknown whether Theorem \ref{main} is true or not if $K/k$ is a MacLane-separable infinite field extension; see \cite{JL82}). \\
	{\rm (3)} Let $k$ be a finite field with $K$ its algebraic closure and $Q$ be a Kronecker quiver. Then $kQ$ is derived-discrete while $KQ$ is not.
\end{rem}

\subsection*{Acknowledgements}
Both authors are grateful to Professor Xiao-Wu Chen for his encouragement and advices. Also, they would like to thank the referees for their helpful suggestions. This work is supported by National Natural Science Foundation of China (Grant Nos. 11961007) and Science Technology Foundation of Guizhou Province (Grant Nos. [2020]1Y405).



\normalsize


\begin{thebibliography}{[HD82]}




\normalsize
\baselineskip=17pt



\bibitem[ASS06]{ASS06} {\sc I. Assem, D. Simson, A. Skowro{\'n}ski}, {\em Elements of the Representation Theory of Associative Algebras}. Volume 1 Techniques of Representation Theory, Cambridge University press, 2006.

\bibitem[Aus74]{Aus74}  {\sc M. Auslander}, {\em Representation theory of artin algebras II}, Comm. Algebra {\bf 2} (1974), 269-310.

\bibitem[Bau07]{Bau07} {\sc R. Bautista}, {\em On derived tame algebras}, Bol. Soc. Mat. Mexicana, {\bf 13} (2007), 25-54.

\bibitem[BD03]{BD03} {\sc V. Bekkert, Yu. Drozd}, {\em Tame-wild dichotomy for derived categories}, arXiv:math.RT/0310352.

\bibitem[BM03]{BM03} {\sc V. Bekkert, H. Merklen}, {\em Indecomposables in derived
categories of gentle algebras}, Alg. Rep. Theory, {\bf 6} (2003), 285-302.

\bibitem[Br41]{Br41}  {\sc R. Brauer}, {\em On the indecomposable representations of algebras}, Bull. Amer. Math. Soc. {\bf 47} (1941),
Abstract 334, Page 684.

\bibitem[Dro86]{Dro86} {\sc Y. A. Drozd}, {\em Tame and wild matrix problems}, in Representations and Quadratic Forms (Institute of Mathematics, Academy of Sciences, Ukrainian SSR, 1979) 39-74; Amer. Math. Soc. Transl. {\bf 128} (1986), 31-55.

\bibitem[DG92]{DG92} {\sc Y. A. Drozd, G. M. Greuel} ,  {\em Tame-wild dichotomy for Cohen-Macaulay modules}, Math. Ann. {\bf 294}(3) (1992), 387-394.

\bibitem[Jans57]{Jans57} {\sc J. P. Jans}, {\em On the indecomposable representations of algebras}, Ann. of Math. {\bf 66} (1957) 418-429.

\bibitem[JL82]{JL82} {\sc C. U. Jensen, H. Lenzing}, {\em Homological dimension and representation type of algebras under base field extension}, Manuscripta Math. {\bf 39} (1982), 1-13.

\bibitem[Ka00]{Ka00} {\sc S. Kasjan}, {\em Auslander-reiten sequences under base field extension}, Proc. Amer. Math. Soc. {\bf 128}(10) (2000), 2885-2896. 

\bibitem[Li19]{Li19} {\sc J. Li}, {\em Algebra extensions and derived-discrete algebras}, arXiv: 1904.07168.

\bibitem[NR75]{NR75} {\sc L. A. Nazarova, A. V. Roiter}, {\em Kategorielle matrizen-probleme und die Brauer-Thrall-vermutung}, Mitt. Math. Sem. Giessen Heft {\bf 115}(1975), 1-153.

\bibitem[Rin80]{Rin80} {\sc C. M. Ringel}, {\em Report on the Brauer-Thrall conjectures: Rojter's theorem and the theorem of Nazarova and Rojter (on algorithms for solving vectorspace problems. I)}. In: Representation Theory I (pp. 104-136). Springer, Berlin, Heidelberg, 1980.

\bibitem[Rin97]{Rin97} {\sc C. M. Ringel}, {\em The repetitive algebra of a gentle algebra}, Bol. Soc. Mat. Mexicana, {\bf 33}(1997), 235-253.

\bibitem[Ro68]{Ro68} {\sc A. V. Roiter}, {\em The unboundedness of the dimensions of the indecomposable representations of algebras that have an infinite number of indecomposable representations}, Izv. Akad. Nauk SSSR Ser.
Math. {\bf 32}(1968), 1275--1282, English transl.: Math. USSR, Izv. 2 (1968), 1223--1230.

\bibitem[Ro78]{Ro78} {\sc A. V. Roiter}, {\em Matrix problems}, Proc. ICM Helsinki, 1978, 319-322.

\bibitem[Sim97]{Sim97} {\sc D. Simson}, {\em Prinjective modules, propartite modules, representations of bocses and lattices over orders}, J. Math. Soc. Japan, {\bf 49}(1997), 31-68. 

\bibitem[Sim05]{Sim05} {\sc D. Simson}, {\em Tame–wild dichotomy of Birkhoff type problems for nilpotent linear operators}, J. Algebra,  {\bf 424}(2005), 254-293.

\bibitem[Th47]{Th47} {\sc R. M. Thrall}, {\em On ahdir algebras}, Bull. Amer. Math. Soc. {\bf 53} (1947), Abstract 22, Page 49

	
\bibitem[Vo01]{Vo01} {\sc D. Vossieck}, {\em The algebras with discrete derived category}, J. Algebra {\bf 243} (2001), 168-176.

\bibitem[Wei95]{Wei95} {\sc C A. Weibel},  An introduction to homological algebra, Cambridge university press, 1995.
	
\bibitem[Zh16a]{Zh16a} {\sc C. Zhang}, {\em On algebras of strongly derived unbounded type}, J. Pure Appl.Algebra {\bf 220} (2016), 1462-1474.

\bibitem[Zh16b]{Zh16b} {\sc C. Zhang}, {\em Derived Brauer-Thrall type theorem I for artin algebras}, Comm. Algebra, {\bf 44}(2016), 3509-3517.
	
\bibitem[ZH16]{ZH16} {\sc C. Zhang, Y. Han}, {\em Brauer-Thrall type theorems for derived module categories},  Alg. Rep. Theory, {\bf 19}(2016), 1369-1386.
\end{thebibliography}
\end{document}